\documentclass[graybox]{svmult}

\usepackage{mathptmx}       
\usepackage{helvet}         
\usepackage{courier}        
\usepackage{type1cm}        
\usepackage{graphicx}        
\usepackage{multicol}        
\usepackage[bottom]{footmisc}

\usepackage{amsmath}

\usepackage{amsthm,amsfonts,amssymb}

\usepackage{mathtools}
\usepackage{verbatim}
\usepackage{bbm}
\usepackage{tikz}
\usepackage{enumerate}
\usepackage{ifthen}
\usepackage{xparse}
\usepackage{multirow}

\usepackage{tikz-cd}

\newcommand{\from}{\mathpunct{:}}
\newcommand{\RR}{\mathbb{R}}
\newcommand{\NN}{\mathbb{N}}
\newcommand{\CC}{\mathbb{C}}
\newcommand{\KK}{\mathbb{K}}

\newcommand{\id}{\mathrm{id}}

\newcommand{\Hopf}{\mathcal{H}}

\newcommand{\one}{\mathbbm{1}}
\newcommand{\tensor}{\otimes}

\newcommand{\Alg}{\mathrm{Alg}}

 \newcommand{\trees}{\mathcal{T}}

\newcommand{\chG}[2]{\mathcal{G}(#1, #2)}
\newcommand{\chA}[2]{\mathfrak{g}(#1, #2)}

   \newcommand{\td}{\diamond}
\newcommand{\Lf}{\ensuremath{\mathbf{L}}}

\newcommand{\LB}[1][\cdot \hspace{1pt} , \cdot]{[\hspace{1pt} #1 \hspace{1pt} ]}
\newcommand{\Frechet}{Fr\'echet }
\newcommand{\norm}[1]{\left\lVert #1 \right\rVert}

\DeclareMathOperator{\Hom}{Hom}

\spnewtheorem{prop}[theorem]{Proposition}{\bfseries}{\itshape}
\spnewtheorem{lem}[theorem]{Lemma}{\bfseries}{\itshape}
\spnewtheorem{cor}[theorem]{Corollary}{\bfseries}{\itshape}
\spnewtheorem{ex}[theorem]{Example}{\bfseries}{\rmfamily}
\spnewtheorem{thm}[theorem]{Theorem}{\bfseries}{\itshape}
\spnewtheorem{defn}[theorem]{Definition}{\bfseries}{\rmfamily}
\spnewtheorem{rem}[theorem]{Remark}{\itshape}{\rmfamily}

\spnewtheorem{setup}{\nocaption}[section]{\bfseries}{\rmfamily}

\tikzstyle dtree=[grow'=up,sibling distance=2mm,level distance=2mm,thick]
\tikzstyle dtree node=[scale=0.3,shape=circle,very thin,draw]
\tikzstyle dtree black node=[style=dtree node,fill=black]
\newcommand{\onenode}{
  \begin{tikzpicture}[dtree]
    \node[dtree black node] {}
    ;
  \end{tikzpicture}
}

\newcommand{\twonode}{
\begin{tikzpicture}[dtree]
  \node[dtree black node] {}
  child { node[dtree black node] {} }
  ;
\end{tikzpicture}
}

\begin{document}

\title*{The geometry of characters of Hopf algebras}
\author{Geir Bogfjellmo and Alexander Schmeding}
\institute{Geir Bogfjellmo \at Chalmers tekniska h\"ogskola och G\"oteborgs universitet, Matematiska vetenskaper,
412 96 G\"oteborg, Sweden \email{geir.bogfjellmo@chalmers.se}
\and Alexander Schmeding \at NTNU Trondheim, Institutt for matematiske fag, 7491 Trondheim, Norway \email{alexander.schmeding@math.ntnu.no}}
\maketitle

\abstract{
Character groups of Hopf algebras can be used to conveniently describe several species of ``series expansions'' such as ordinary Taylor series, B-series, arising in the study of ordinary differential equations, Fliess series, arising from
control theory and rough paths.
These ideas are a fundamental link connecting Hopf algebras and their character groups to the topics of the Abelsymposium 2016 on ``Computation and Combinatorics in Dynamics, Stochastics and Control''.
In this note we will explain some of these connections, review constructions for Lie group and topological structures for character groups and provide some new results for character groups.
Our main result is the construction and study of Lie group structures for Hopf algebras which are graded but not necessarily connected (in the sense that the degree zero subalgebra is finite-dimensional).}\medskip

\keywords{infinite-dimensional Lie group, pro-Lie group, character group of a Hopf algebra, regularity of Lie groups, continuous inverse algebra\\[1em]
 (Secondary)}\medskip
 
\setcounter{minitocdepth}{1}
\dominitoc
\newpage

Character groups of Hopf algebras appear in a variety of mathematical contexts.
For example, they arise in non-commutative geometry, renormalisation of quantum field theory \cite{MR2371808}, numerical analysis \cite{Brouder-04-BIT} and the theory of regularity structures for stochastic partial differential equations \cite{MR3274562}.
A Hopf algebra is a structure that is simultaneously a (unital, associative) algebra, and a (counital, coassociative) coalgebra that is also equipped with an antiautomorphism known as the antipode, satisfying a certain property.
In the contexts of these applications, the Hopf algebras often encode combinatorial structures and serve as a bookkeeping device.

Several species of ``series expansions'' can then be described as algebra morphisms from a Hopf algebra $\Hopf$ to a commutative algebra $B$.
Examples include ordinary Taylor series, B-series, arising in the study of ordinary differential equations, Fliess series, arising from
control theory and rough paths, arising in the theory of stochastic ordinary equations and partial differential equations.
An important fact about such algebraic objects is that, if $B$ is commutative, the set of algebra morphisms $\Alg(\Hopf, B)$, also called \emph{characters}, forms a group with product given by convolution
\[a\ast b = m_B \circ (a\tensor b)\circ \Delta_\Hopf.
\]
These ideas are the fundamental link connecting Hopf algebras and their character groups to the topics of the Abelsymposium 2016 on ``Computation and Combinatorics in Dynamics, Stochastics and Control''.
In this note we will explain some of these connections, review constructions for Lie group and topological structures for character groups and provide some new results for character groups.

Topological and manifold structures on these groups are important to applications in the various fields outlined above.
In many places in the literature the character group is viewed as ``an infinite dimensional Lie group'' and one is interested in solving differential equations on these infinite-dimensional spaces (we refer to \cite{BDS17} for a more detailed discussion and further references).
This is due to the fact that the character group admits an associated Lie algebra, the Lie algebra of infinitesimal characters\footnote{Note that this Lie algebra is precisely the one appearing in the famous Milnor-Moore theorem in Hopf algebra theory \cite{MR0174052}.}
$$\chA{\Hopf}{B} := \{\phi \in \text{Hom}_\KK (\Hopf,B) \mid \phi (xy) = \phi (x) \varepsilon_\Hopf (y) + \varepsilon_\Hopf (x) \phi (y) , \ \forall x,y \in \Hopf \},$$
whose Lie bracket is given by the commutator bracket with respect to convolution.
As was shown in \cite{BDS16}, character groups of a large class of Hopf algebras are infinite-dimensional Lie groups. Note however, that in ibid.\ it was also shown that not every character group can be endowed with an infinite-dimensional Lie group structure.
In this note we extend these results to a larger class Hopf algebras. To this end, recall that a topological algebra is a \emph{continuous inverse algebra} (or CIA for short) if the set of invertible elements is open and inversion is continuous on this set (e.g.\ a Banach algebra). Then we prove the following theorem.
\smallskip

\textbf{Theorem A} \emph{Let $\Hopf = \oplus_{n \in \NN_0} \Hopf_n$ be a graded Hopf algebra such that $\mathrm{dim}\ \Hopf_0 < \infty$ and $B$ be a commutative CIA. Then $\chG{\Hopf}{B}$ is an infinite-dimensional Lie group whose Lie algebra is $\chA{\Hopf}{B}$.}\smallskip

As already mentioned, in applications one is interested in solving differential equations on character groups (see e.g.\ \cite{MR3485151} and compare \cite{BDS17}).
These differential equations turn out to be a special class of equations appearing in infinite-dimensional Lie theory in the guise of regularity for Lie groups.
To understand this and our results, we recall this concept now for the readers convenience.
\smallskip

\textbf{Regularity (in the sense of Milnor)} 
Let $G$ be a Lie group modelled on a locally convex space, with identity element $e$, and
 $r\in \NN_0\cup\{\infty\}$. We use the tangent map of the left translation
 $\lambda_g\colon G\to G$, $x\mapsto gx$ by $g\in G$ to define
 $g.v:= T_{e} \lambda_g(v) \in T_g G$ for $v\in T_{e} (G) =: \Lf(G)$.
 Following \cite{1208.0715v3}, $G$ is called
 \emph{$C^r$-semiregular} if for each $C^r$-curve
 $\gamma\colon [0,1]\rightarrow \Lf(G)$ the initial value problem
 \begin{displaymath}
  \begin{cases}
   \eta'(t)&= \eta (t). \gamma(t)\\ \eta(0) &= e
  \end{cases}
 \end{displaymath}
 has a (necessarily unique) $C^{r+1}$-solution $\text{Evol} (\gamma):=\eta\colon [0,1]\rightarrow G$.
 If furthermore the map
 \begin{displaymath}
  \text{evol} \colon C^r([0,1],\Lf(G))\rightarrow G,\quad \gamma\mapsto \text{Evol}
  (\gamma)(1)
 \end{displaymath}
 is smooth, $G$ is called \emph{$C^r$-regular}.\footnote{Here we consider $C^r([0,1],\Lf(G))$ as a locally convex vector space with the pointwise operations and the topology of uniform convergence of the function and its derivatives on compact sets.} If $G$ is $C^r$-regular and $r\leq s$, then $G$ is also
 $C^s$-regular. A $C^\infty$-regular Lie group $G$ is called \emph{regular}
 \emph{(in the sense of Milnor}) -- a property first defined in \cite{MR830252}.
 Every finite-dimensional Lie group is $C^0$-regular (cf.\ \cite{MR2261066}).
 In the context of this paper our results on regularity for character groups of Hopf algebras subsume the following theorem.
\smallskip

\textbf{Theorem B} \emph{Let $\Hopf = \oplus_{n \in \NN_0} \Hopf_n$ be a graded Hopf algebra such that $\mathrm{dim }\ \Hopf_0 < \infty$ and $B$ be a sequentially complete commutative CIA. Then $\chG{\Hopf}{B}$ is $C^0$-regular.}\smallskip

 Recently, also an even stronger notion regularity called measurable regularity has been considered \cite{1601.02568v1}.
 For a Lie group this stronger condition induces many Lie theoretic properties (e.g.\ validity of the Trotter product formula).
 In this setting, $L^1$-regularity means that one can solve the above differential equations for absolutely continuous functions (whose derivatives are $L^1$-functions).
 A detailed discussion of these properties can be found in \cite{1601.02568v1}. However, we will sketch in Remark \ref{rem: L1reg} a proof for the following proposition.\smallskip

\textbf{Proposition C}  \emph{Let $\Hopf = \oplus_{n \in \NN_0} \Hopf_n$ be a graded Hopf algebra with $\mathrm{dim}\ \Hopf_0 < \infty$ which is of countable dimension, e.g.\ $\Hopf$ is of finite type. Then for any commutative Banach algebra $B$, the group $\chG{\Hopf}{B}$ is $L^1$-regular.}\smallskip

One example of a Hopf algebra whose group of characters represent a series expansion is the Connes--Kreimer Hopf algebra or Hopf algebra of rooted trees $\Hopf_{\mathrm{CK}}$.

Brouder \cite{Brouder-04-BIT} established a very concrete link between $\Hopf_{\mathrm{CK}}$ and B-series.
B-series, due to Butcher \cite{Butcher72}, constitute an algebraic structure for the study of integrators for ordinary differential equations. In this context, the group of characters $\chG{\Hopf_{\mathrm{CK}}}{\RR}$ is known as the Butcher group.
The original idea was to isolate the numerical integrator from the concrete differential equation, and even from the surrounding space (assuming only that it is affine), thus enabling a study of the integrator \emph{an sich}.
\smallskip

Another example connecting character groups to series expansions arises in the theory of regularity structures for stochastic partial differential equations (SPDEs) \cite{MR3274562,1610.08468v1}.
In this theory one studies singular SPDEs, such as the continuous parabolic Anderson model (PAM, cf.\ the monograph \cite{MR3526112}) formally given by
$$ \left(\frac{\partial}{\partial t} - \Delta\right) u(t,x) = u(t,x) \zeta(x) \quad (t,x) \in ]0,\infty[ \times \RR^2, \quad \zeta\  \text{spatial white noise}.$$
We remark that due to the distributional nature of the noise, the product and thus the equation is ill-posed in the classical sense (see \cite[p.5]{MR3274562}).
To make sense of the equation, one wants to describe a potential solution by ``local Taylor expansion'' with respect to reference objects built from the noise terms.
The analysis of this ``Taylor expansion'' is quite involved, since products of noise terms are not defined.
However, it is possible to obtains Hopf algebras which describe the combinatorics involved. Their $\RR$-valued character group $\mathcal{G}$ is then part of a so called regularity structure $(\mathcal{A}, \mathcal{T}, \mathcal{G})$ (\cite[Definition 5.1]{1610.08468v1}) used in the renormalisation of the singular SPDE. See Example \ref{ex: Hairer} for a discussion of the Lie group structure for these groups.

\section{Foundations: Character groups and graded algebra}

In this section we recall basic concepts and explain the notation used throughout the article.
Whenever in the following we talk about algebras (or coalgebras or bialgebras) we will assume that the algebra (coalgebra, bialgebra) is unital (counital or unital and counital in the bialgebra case). 
Further $\KK \in \{\RR, \CC\}$ will always denote either the field of real or complex numbers (though many of the following concepts make sense over general base fields).

\begin{defn}
A \emph{Hopf algebra} (over $\KK$) $\Hopf$ is a $\KK-$bialgebra $(\Hopf, m,\one_\Hopf, \Delta, \varepsilon)$ equipped with an antiautomorphism $S$, called the \emph{antipode}, such that the diagram
\[
\begin{tikzcd}[row sep=4em]
& H\otimes H \arrow[rr,"S\otimes\id"] && H\otimes H \arrow[dr,"m"] \\
H \arrow[ur,"\Delta"] \arrow[rr,"\varepsilon"] \arrow[dr,"\Delta"'] && K \arrow[rr,"u"] && H \\
& H\otimes H \arrow[rr,"\id\otimes S"'] && H\otimes H \arrow[ur,"m"']
\end{tikzcd}
\]
commutes.

In the diagram $u \colon \KK \rightarrow \Hopf, k \mapsto k\one_\Hopf$ is the unit map of $\Hopf$, i.e.\ the map which sends scalars to multiples of the unit $\one_\Hopf \in \Hopf$. We refer to \cite{MR1381692,MR2290769,manchon,Swe69} for basic information on Hopf algebras.
\end{defn}

Let $B$ be a commutative algebra. The set of linear maps $\Hom_\KK (\Hopf, B)$ forms a new algebra with the convolution product
\[
\phi\star\psi = m_B \circ (\phi\tensor \psi) \Delta_\Hopf,
\]
and unit $u_B \circ \varepsilon_\Hopf$ (where $u_B$ is the unit map of $B$).

Recall that the invertible elements or \emph{units} of an algebra $A$ form a group, which we denote $A^\times$.

\begin{defn}\label{defn: char}
A linear map
$\phi \from \Hopf \to B$ is called
\begin{enumerate}
\item a ($B$-valued) \emph{character} if $\phi(ab) = \phi(a)\phi(b)$ for all $a,b\in \Hopf$.
        The set of all characters is denoted $\chG{\Hopf}{B}$.
\item a ($B$-valued) \emph{infinitesimal character} if $\phi(ab) = \varepsilon_\Hopf(b) \phi(a)  + \varepsilon_\Hopf(a) \phi(b)$ for all $a,b\in \Hopf$.
    The set of all infinitesimal characters is denoted $\chA{\Hopf}{B}$.
\end{enumerate}
\end{defn}

\begin{lem}[{\cite[Proposition 21 and 22]{manchon}}]
\begin{enumerate}
\item $\chG{\Hopf}{B}$ is a subgroup of the group of units $\Hom_\KK (\Hopf, B)^\times$. On $\chG{\Hopf}{B}$, the inverse is given by
    \[\phi^{\star-1}= \phi \circ S
    \]
\item $\chA{\Hopf}{B}$ is a Lie subalgebra of the commutator Lie algebra $\Hom_\KK (\Hopf, B), [\cdot, \cdot]$, where the bracket is given by
\[[\phi, \psi] = \phi\star \psi - \psi \star \phi
\]
\end{enumerate}
\label{lem: LGLA}
\end{lem}

An algebraic property of characters and infinitesimal characters is that the algebraic exponential
\[
\exp^{\star}(\phi) = \sum_{n=0}^{\infty} \frac{1}{n!} \phi^{\star n}
\]
is a map from  $\chA{\Hopf}{B}$ to $\chG{\Hopf}{B}$. \cite[Proposition 22]{manchon}.

In order to study the topological aspects of characters and infinitesimal characters of Hopf algebras, we need to assume at this step that $B$ is a topological algebra, i.e., that $B$ is topological vector space and that the product in $B$ is a continuous bilinear function
\[ \mu_{B} \from B \times B \to B
\]
We can then endow the space $\Hom_{\KK}(\Hopf, B)$ with the \emph{topology of pointwise convergence}.
The sets $\chA{\Hopf}{B}$ and $\chG{\Hopf}{B}$ are then closed subsets of $\Hom_{\KK}(\Hopf, B)$, and carry the induced topologies.

\begin{prop}\label{prop: topgp}
Let $\Hopf$ be a Hopf algebra, and $B$ a commutative, topological algebra. Endow $\Hom_{\KK}(\Hopf, B)$ with the topology of pointwise convergence. Then
\begin{itemize}
\item $(\Hom_{\KK}(\Hopf, B), \star)$ is a topological algebra,
\item $\chG{\Hopf}{B}$ is a topological group,
\item $\chA{\Hopf}{B}$ is a topological Lie algebra.
\end{itemize}
\end{prop}
\begin{proof} It is sufficient to prove that $\star$ is continuous. Since $\Hom_{\KK}(\Hopf, B)$ is endowed with the topology of pointwise convergence, it suffices to test convergence when evaluating at an element $h\in \Hopf$. Using Sweedler notation, we get $\phi\star \psi(h) = \sum_{(h)} \phi (h_{(1)}) \star \psi(h_{(2)})$ where the multiplication is carried out in $B$. As point evaluations are continuous on $\Hom_{\KK}(\Hopf, B)$, and multiplication is continuous in $B$, $\star$ is continuous.
\end{proof}

The definition of $\star$ does not depend on the algebra structure of $\Hopf$, only the coalgebra structure.
We therefore get as a corollary:
\begin{cor}
Let $C$ be a coalgebra, and $B$ a commutative, topological algebra. Then $(\Hom_{\KK}(C, B), \star)$, equipped with the topology of pointwise convergence, is a topological algebra.
\end{cor}

In Section \ref{sec: geometry} we will be able to state more about the topology and geometry of groups of characters, under further assumptions on $\Hopf$ and $B$.
In particular, we are interested in cases where $\chG{\Hopf}{B}$ is an (infinite dimensional) Lie group, or a projective limit of finite dimensional Lie groups, i.e. a \emph{pro-Lie} group.
Both of these classes of topological groups can to some extent claim to be the generalization of finite dimensional Lie groups, and have been studied extensively for this reason (see e.g.\ \cite{HMax,MR2261066,MR2337107}).

For many arguments later on gradings will be important. We recall the following basic definitions
\begin{defn}\label{setup: graded}
 Recall that a graded Hopf algebra $\Hopf = \bigoplus_{n \in \NN_0} \Hopf_0$ is a Hopf algebra together with a grading as algebra and coalgebra (i.e.\ $\Hopf_n \cdot \Hopf_m \subseteq \Hopf_{n+m}$ and $\Delta (\Hopf_n) \subseteq \bigoplus_{k+l =n} \Hopf_k \otimes \Hopf_l$). In particular, $\Hopf_0$ becomes a Hopf subalgebra of $\Hopf$.

 Note that for a graded Hopf algebra $\Hopf$, identifying a mapping $f \colon \Hopf \rightarrow B$ with its components on the grading induces a natural isomorphisms of topological vector spaces (with respect to the topologies of pointwise convergence)
  \begin{displaymath}
   \Hom_\KK(\Hopf, B) = \Hom_\KK (\bigoplus_{n \in \NN_0} \Hopf_n , B) \cong \prod_{n \in \NN_0} \Hom_{\KK} (\Hopf_n , B).
  \end{displaymath}
 Hence $A = \Hom_{\KK} (\Hopf, B)$ becomes a densely graded topological vector space (see \cite[Appendix B]{BDS16}).
 We denote by $A_n := \Hom_\KK (\Hopf_n, B)$ the parts of the dense grading.
 Note that $A_0$ becomes a locally convex subalgebra of $A$ by definition of the grading.
\end{defn}

\section{Geometry of groups of characters}
\label{sec: geometry}
In this section, we review results on geometry, topology and Lie theory for character groups of Hopf algebras $\chG{\Hopf}{B}$.
Further, in Section \ref{sect: noncon} we prove a new result which establishes a Lie group structure for character groups of non-connected Hopf algebras.

In general, the existence of a Lie group structure on the character group of a Hopf algebra $\Hopf$ depends on structural properties of the underlying Hopf algebra (e.g.\ we need graded and connected Hopf algebras), the table below provides an overview of the topological and differentiable structures related to these additional assumptions.

\begin{figure}\label{fig1}
\begin{center}
 \begin{tabular}{|c|c|c|}
 \hline
 Hopf algebra $\Hopf$ & commutative algebra $B$ & Structure on $\chG{\Hopf}{B}$ \\
 \hline
 arbitrary            & weakly complete & pro-Lie group (cf.\ Remark \ref{rem: proLie})\\
 graded and dim $\Hopf_0 < \infty$ & continuous inverse algebra & $\infty$-dim. Lie group (Section \ref{sect: noncon})\\
 graded and connected & locally convex algebra & $\infty$-dim. Lie group (Section \ref{sect: conn})\\
 \hline
 \end{tabular}
 \caption{Overview of topological and Lie group structures on character groups of Hopf algebras.}
 \end{center}
\end{figure}

In general, the character group need not admit a Lie group structure as was shown in \cite[Example 4.11]{BDS16}. There we exhibited a character group of the group algebra of an abelian group of infinite rank which can not be a Lie group.

\begin{rem}\label{rem: proLie}
If the target algebra $B$ is a weakly complete algebra, e.g.\ a finite dimensional algebra, the character group $\chG{\Hopf}{B}$ is always a projective limit of finite-dimensional Lie groups.
In \cite[Theorem 5.6]{BDS16} we have proved that for an arbitrary Hopf algebra $\Hopf$ and $B$ a weakly complete algebra, $\chG{\Hopf}{B}$ is a special kind of topological group, a so called pro-Lie group (see the survey \cite{HMax}).
A pro-Lie group is closely connected to its pro-Lie algebra which turns out to be isomorphic to $\chA{\Hopf}{B}$ for the pro-Lie group $\chG{\Hopf}{B}$.
Although pro-Lie groups do not in general admit a differentiable structure, a surprising amount of Lie theoretic properties carries over to pro-Lie groups (we refer to the monograph \cite{MR2337107} for a detailed account).
 \end{rem}

 Often the character group of a Hopf algebra will have more structure than a topological group.
 As we will see in the next example character groups often carry Lie group structures.

\begin{ex}
 Let $G$ be a compact Lie group. Then we consider the set $\mathcal{R} (G)$ of
\emph{representative functions}, i.e.\ continuous functions $f \colon G
\rightarrow \RR$ such that the set of right translates $R_x f \colon G
\rightarrow \RR, \ R_x f(y) = f(yx)$ generates a finite-dimensional subspace of
$C(G,\RR)$, cf.\ \cite[Chapter 3]{MR3114697} or \cite[Section 3]{MR2290769} for more information and
examples.

 Using the group structure of $G$ and the algebra structure induced by
$C(G,\RR)$ (pointwise multiplication), $\mathcal{R} (G)$ becomes a Hopf algebra
(see \cite[pp. 42-43]{MR1789831}).
 Following Remark \ref{rem: proLie}, we know that $\chG{\mathcal{R} (G), \RR}$ becomes a topological group.

 It follows from Tannaka-Kre\u{i}n duality that as compact groups
$\chG{\mathcal{R} (G)}{\RR}\cong G$, whence $\chG{\mathcal{R} (G), \RR}$
inherits a posteriori a Lie group structure  \cite[Theorem 1.30 and
1.31]{MR1789831}.\footnote{This is only a glimpse at Tannaka-Kre\u{i}n duality,
which admits a generalisation to compact topological groups (using complex
representative functions, see \cite[Chapter 7, \S 30]{MR0262773} and cf.\
\cite[p. 46]{MR1789831} for additional information in the Lie group case). Also
we recommend \cite[Chapter 6]{MR3114697} for more information on compact Lie
groups.} Observe that the Lie group structure induced on $\chG{\mathcal{R} (G), \RR}$ via Tannaka-Kre\u{i}n duality coincides with the ones discussed in Sections \ref{sect: noncon} and \ref{sect: conn} (when these results are applicable to the Hopf algebra of representative functions).
\end{ex}

\begin{ex}[The Butcher group]
Let $\trees$ denote the set of rooted trees, and $\Hopf_{\mathrm{CK}}=\langle\! \langle \trees \rangle\!\rangle$ the free commutative algebra generated by $\trees$.
The Grossman--Larson coproduct is defined on trees as
\[
\Delta(\tau) = \tau \tensor \one + \sum_{\sigma} (\tau \setminus \sigma) \tensor \sigma
\]
where the sum goes over all connected subsets $\sigma$ of $\tau$ containing the root.
Together with the algebra structure, The Grossman--Larson coproduct defines a graded, connected bialgebra structure on $\Hopf_{\mathrm{CK}}$, and therefore also a Hopf algebra structure.

The characters $\chG{\Hopf_{\mathrm{CK}}}{\RR}$ are the algebra morphisms $\Hom_{\Alg}(\Hopf_{\mathrm{CK}}, \RR)$.
Clearly, we can identify
\[ \chG{\Hopf_{\mathrm{CK}}}{\RR} \simeq\RR^{\trees}
\]

In numerical analysis, the character group $\chG{\Hopf_{\mathrm{CK}}}{\RR}$ is known as the Butcher group \cite{Butcher72, HLW2006}.
This group is closely connected to a class of numerical integrators for ordinary differential equations.
Namely, we let $\dot{y} = f(y)$ be an autonomous ordinary differential equation on an affine space $E$.
Many numerical integrators \footnote{To be exact: The class of integrators depending only on the affine structure (cf.\ \cite{MR3510021}).}  can be expanded in terms of the \emph{elementary differentials} of the vector field $f$, i.e. as a series
\begin{equation}
 y_{n+1} = y_n + a(\onenode) hf(y_n) + a(\twonode) h^2 f^{'}f(y_n) +\dotsb
 \label{eq: bseries}
 \end{equation}
The elementary differentials are in a natural one-to-one correlation with $\trees$, and the series \eqref{eq: bseries} thus defines an element in $\chG{\Hopf_{\mathrm{CK}}}{\RR}$.
The crucial observation is that, (after a suitable scaling,) the convolution product in $\chG{\Hopf_{\mathrm{CK}}}{\RR}$ corresponds to the composition of numerical integrators.

In the following, it will be established that $\chG{\Hopf_{\mathrm{CK}}}{\RR}$ is a $\RR$-analytic, $C^0$-regular Fr\'{e}chet Lie group as well as a pro-Lie group. See \cite{BS15, BDS16} for further details on the Lie group structure of $\chG{\Hopf_{\mathrm{CK}}}{\RR}$.

However, in some sense, the Butcher group $\chG{\Hopf_{\mathrm{CK}}}{\RR}$ is too big to properly analyze numerical integrators.
For every numerical integrator, the coefficients $a\from \trees \to \RR$ satisfy a growth bound $|a(\tau)|\le CK^{|\tau|}$.
Elements of  $\chG{\Hopf_{\mathrm{CK}}}{\RR}$ satisfying such growth bounds form a subgroup, and even a Lie subgroup.
However, the modelling space now becomes a Silva space\footnote{Silva spaces arise as special inductive limits of Banach spaces, see \cite{MR0287271} for more information. They are also often called (DFS)-space in the literature, as they can be characterised as the duals of Fr\'{e}chet-Schwartz spaces.}.
This group is studied in the article \cite{BS16tame}.
\end{ex}

 In the next section we begin with establishing general results on the infinite-dimensional Lie group structure of Hopf algebra character groups.
 These manifolds will in general be modelled on spaces which are more general then Banach spaces.
 Thus the usual differential calculus has to be replaced by the so called Bastiani calculus (see \cite{MR0177277}, i.e.\ differentiability means existence and continuity of directional derivatives).
 For the readers convenience, Appendix \ref{app: calculus} contains a brief recollection of this calculus.

%

\subsection{Character groups for $\Hopf$ graded with finite dimensional $\Hopf_0$ and $B$ a continuous inverse algebra}\label{sect: noncon}

In this section we consider graded but not necessarily connected Hopf algebras.
In general, character groups of non-connected Hopf algebras do not admit a Lie group structure.
Recall from example from \cite[Example 4.11 (b)]{BDS16} that the character group of the group algebra of an infinite-group does in general not admit a Lie group structure.
However, if the Hopf algebra is not too far removed from being connected (i.e.\ the $0$-degree subspace $\Hopf_0$ is finite-dimensional) and the target algebra is at least a continuous inverse algebra, we can prove that the character group $\chG{\Hopf}{B}$ is an infinite-dimensional Lie group.
This result is new and generalises \cite{BDS16} where only character groups of graded and connected Hopf algebras were treated (albeit the target algebra in the graded connected case may be a locally convex algebra).

\begin{setup}
 Let $(A,\cdot)$ be a (real or complex) locally convex algebra (i.e.\ the locally convex topology of $A$ makes the algebra product jointly continuous). We call $(A,\cdot)$ \emph{continuous inverse algebra} (or \emph{CIA} for short) if its unit group $A^\times$ is open and inversion $A^\times \rightarrow A, a \mapsto a^{-1}$ is continuous.
\end{setup}
The class of locally convex algebras which are CIAs are of particular interest to infinite-dimensional Lie theory as their unit groups are in a natural way (analytic) Lie groups (see \cite{MR1948922,MR2997582}).

Before we establish the structural result, we have to construct an auxiliary Lie group structure in which we can realise the character group as subgroup.
To construct the auxiliary Lie group, we can dispense with the Hopf algebra structure and consider only (graded) coalgebras at the moment.
The authors are indebted to K.--H.\ Neeb for providing a key argument in the proof of the following Lemma.

\begin{lem}\label{lem: cia0}
 Let $(C, \Delta)$ be a finite-dimensional coalgebra and $B$ be a CIA. Then $(\Hom_\KK (C,B), \star)$ is a CIA.
\end{lem}
\begin{proof}
 Consider the algebra $(A := \Hom_{\KK} (C,\KK), \star_A)$, where $\star_A$ is the convolution product.
 Then the algebraic tensor product $T := B \otimes_\KK A$ with the product $(b\otimes \varphi) \cdot (c\otimes \psi) := bc \otimes \varphi \star_A \psi$ is a locally convex algebra.
 Since $C$ is a finite-dimensional coalgebra, $A$ is a finite-dimensional algebra. Due to an argument which was communicated to the authors by K.--H.\ Neeb, the tensor product of a finite-dimensional algebra and a CIA is again a CIA.\footnote{We are not aware of a literature reference of this fact apart from the forthcoming book by Gl\"{o}ckner and Neeb \cite{GNforth}. To roughly sketch the argument from \cite{GNforth}: Using the regular representation of $A$ one embeds $B \otimes_\KK A$ in the matrix algebra $M_n (B)$ (where $n = \text{dim } A$). Now as $A$ is a commutant in $\text{End}_\KK (A)$, the same holds for $B \otimes A$ in $M_n (B)$. The commutant of a set in a CIA is again a CIA, whence the assertion follows as matrix algebras over a CIA are again CIAs (cf.\ \cite[Corollary 1.2]{MR0448350}).}
 Thus it suffices to prove that the linear mapping defined on the elementary tensors via
 $$\kappa \colon T \rightarrow \Hom_{\KK} (C,B) , \ b \otimes \varphi \mapsto (x \mapsto \varphi(x)b)$$
 is an isomorphism of unital algebras.
 Since $A$ is finite-dimensional, it is easy to see that $\kappa$ is an isomorphism of locally convex spaces.
 Thus it suffices to prove that $\kappa$ is a algebra morphism.
 To this end let $\varepsilon$ be the counit of $C$. We recall that $\one_A = \varepsilon$, $\one_{\Hom_\KK (C,B)} = (x \mapsto \varepsilon(x) \cdot \one_B)$ and $\one_T = \one_B \otimes \one_1 = \one_B \otimes \varepsilon$ are the units in $A$, $\Hom_\KK (C,B)$ and $T$, respectively.
 Now $\kappa (\one_T ) = (x \mapsto \varepsilon (x) \one_B) = \one_{\Hom_\KK (C,B)}$, whence $\kappa$ preserves the unit.

 As the elementary tensors span $T$, it suffices to establish multiplicativity of $\kappa$ on elementary tensors $b_1 \otimes \varphi,  b_2\otimes \psi \in T$.
 For $c \in C$ we use Sweedler notation to write $\Delta (c) = \sum_{(c)} c_1 \otimes c_2$. Then
 \begin{align*}
  \kappa ((b_1\otimes \varphi) \cdot (b_2 \otimes \psi))(c) &= \kappa (b_1b_2 \otimes \varphi \star_A \psi)(c) =  \varphi \star_A \psi (x) b_1 b_2 \\ &=  \sum_{(c)} \varphi(c_1) \psi (c_2) b_1 b_2 = \sum_{(c)} (\varphi (c_1) b_1) (\psi(c_2)b_2) \\
  &= \sum_{(c)} \kappa (b_1 \otimes \varphi) (c_1) \kappa (b_2 \otimes \psi)(c_2)\\ &= \kappa (b_1 \otimes \varphi) \star \kappa (b_2 \otimes \psi)(c)
 \end{align*}
 shows that the mappings agree on each $c \in C$, whence $\kappa$ is multiplicative. Summing up, $\Hom_\KK (C,B)$ is a CIA as it is isomorphic to the CIA $B \otimes A$.
 \end{proof}

 \begin{prop}[$A^\times$ is a regular Lie group]				\label{prop: unit group Lie group}
  Let $C = \bigoplus_{n \in \NN_0} C_n$ be a graded coalgebra with $\mathrm{dim }\  C_0 < \infty$ and $B$ be a CIA.
  Then $A = (\Hom_{\KK}(\Hopf,B),\star)$ is a CIA whose unit group $A^\times$ is Baker--Campbell--Hausdorff--Lie group (BCH--Lie group)\footnote{BCH-Lie groups derive their name from the fact that there is a neighborhood in their Lie algebra on which the Baker--Campbell--Hausdorff series converges and yields an analytic map. See Definition \ref{defn: BCH}.} with Lie algebra $(A,\LB )$, where $\LB $ denotes the commutator bracket with respect to $\star$.

  If in addition $B$ is Mackey complete, then $A$ is Mackey complete and the Lie group $A^\times$ is $C^1$-regular. If $B$ is even sequentially complete, so is $A$ and the Lie group $A^\times$ is $C^0$-regular.
  In both cases the associated evolution map is even $\KK$-analytic and the Lie group exponential map is given by the exponential series.
 \end{prop}

 \begin{proof}
  Recall from \cite[Lemma 1.6 (c) and Lemma B.7]{BDS16} that the locally convex algebra $A$ is a Mackey complete CIA since $A_0$ is such a CIA by Lemma \ref{lem: cia0} (both CIAs are even sequentially complete if $B$ is so).
  Now the assertions concerning the Lie group structure of the unit group are due to Gl\"{o}ckner (see \cite{MR1948922}).

  To see that the Lie group $A^\times$ is $C^k$-regular ($k=1$ for Mackey complete and $k=0$ of sequentially complete CIA $B$), we note that the regularity of the unit group follows from the so called called (GN)-property (cf.\ \cite{BDS16} and see Definition \ref{defn: GN_property} below) and (Mackey) completeness of $A$.
  Having already established completeness, we recall from \cite[Lemma 1.10]{BDS16} that $A$ has the (GN)-property if $A_0$ has the (GN)-property.
  Below in Lemma \ref{lem: GN-Banach0} we establish that $A_0$ has the (GN)-property if $B$ has the (GN)-property.
  Now $B$ is a commutative Mackey complete CIA, whence $B$ has the (GN)-property by \cite[Remark 1.2 and the proof of Corollary 1.3]{MR2997582}.
  Summing up, $A$ has the (GN)-property and thus \cite[Proposition 4.4]{MR2997582} asserts that $A^\times$ is $C^k$-regular with analytic evolution map.
 \end{proof}

 Before we can establish the (GN)-property for $A_0$ as in the proof of Proposition \ref{prop: unit group Lie group} we need to briefly recall this condition.
  \begin{defn}[(GN)-property]								\label{defn: GN_property}
  A locally convex algebra $A$ is said to satisfy the \emph{(GN)-property}, if for every continuous seminorm $p$ on $A$, there exists a continuous seminorm $q$ and a number $M\geq0$ such that for all $n\in\NN$, we have the estimate:
  \begin{equation}\label{eq: GN}
   \norm{\mu_A^{(n)}}_{p,q} := \sup \{p(\mu_A^{(n)} (x_1, \ldots ,x_n)) \mid q(x_i) \leq 1,\ 1 \leq i \leq n\} \leq M^n.
  \end{equation}
  Here, $\mu_A^{(n)}$ is the $n$-linear map
  $\mu_A^{(n)} \colon \underbrace{A\times\cdots\times A}_{n \text{ times}} \rightarrow A, (a_1,\ldots,a_n) \mapsto a_1 \cdots a_n$.
  \end{defn}

\begin{lem}\label{lem: GN-Banach0}
 Let $C$ be a finite-dimensional coalgebra and $B$ be a CIA with the (GN)-property. Then the following holds:
 \begin{enumerate}
  \item $(A := \Hom_\KK (C,B),\star)$ has the (GN)-property.
  \item If $(B, \norm{\cdot})$ is even a Banach algebra, then $A$ is a Banach algebra.
 \end{enumerate}
\end{lem}

\begin{proof}
 Choose a basis $e_i , 1 \leq i \leq d$ for $\Hopf_0$. Identifying linear maps with the coefficients on the basis, we obtain $A = \Hom_\KK (C,B) \cong B^d$ (isomorphism of topological vector spaces).
 For every continuous seminorm $p$ on $B$, we obtain a corresponding seminorm $p_\infty \colon A \rightarrow \RR, \phi \mapsto \max_{1\leq i \leq d} p (\phi (e_i))$ and these seminorms form a generating set of seminorms for the locally convex algebra $A$.
 Let us now write the coproduct of the basis elements as
 $$\Delta (e_i) = \sum_{j,k} \nu_i^{jk} e_j \otimes e_k, \quad 1\leq i\leq d \quad \text{for } \nu_i^{jk} \in \KK.$$
 To establish the properties of $A$ we need an estimate of the structural constants, whence we a constant $K := d^2 \max_{i,j,k} \{|\nu_{i}^{jk}| ,1\}$
 \begin{enumerate}
 \item It suffices to establish the (GN)-property for a set of seminorms generating the topology of $A$. Hence it suffices to consider seminorms $q_\infty$ induced by a continuous seminorm $q$ on $B$.
 Since $B$ has the (GN)-property there are a continuous seminorm $p$ on $B$ and a constant $M \geq 0$ which satisfy \eqref{eq: GN} with respect to the chosen $q$. We will now construct a constant such that \eqref{eq: GN} holds for the seminorms $q_\infty$ and $p_\infty$ taking the r\^{o}le of $q$ and $p$.

 Observe that $q_\infty (\psi) \leq 1$ implies that $q(\psi(e_i)) \leq 1$ for each $1\leq i \leq d$.
 Thus by choice of the constants a trivial computation shows that the constant $KM$ satisfies \eqref{eq: GN} for $q_\infty, p_\infty$ and $n=1$.
 We will show that this constant satisfies the inequality also for all $n>1$ and begin for the readers convenience with the case $n=2$.
  Thus for $\psi_1,\psi_2 \in A$ with $q_\infty (\psi_l) \leq 1, l=1,2$ and $1\leq k \leq d$ we have
 \begin{align*}
  p (\psi_1 \star \psi_2 (e_i)) &= p \left(\sum_{j,k} \nu^{jk}_i \psi_1 (e_j) \psi_2 (e_k) \right) \leq \sum_{j,k} |\nu^{jk}_i| \underbrace{p(\psi_1 (e_i)\psi_2(e_j))}_{\leq M^2}\\
   &\leq \underbrace{K}_{\geq 1}M^2\leq (KM)^2
 \end{align*}
 As the estimate neither depends on $i$ nor on the choice of $\psi_1, \psi_2$ (we only used $q_\infty (\psi_l) \leq 1$), $KM$ satisfies $\norm{\mu_A^{(2)}}_{p_\infty,q_\infty} \leq (KM)^2$ .
 Now for general $n \geq 2$ we choose $\psi_l \in A$ with $q_\infty (\psi_l) \leq 1$ and $1\leq l \leq n$.
 As convolution is associative, $\psi_1 \star \cdots \star \psi_n$ is obtained from applying $\psi_1 \otimes \cdots \otimes \psi_n$ to the iterated coproduct $\Delta^n := \id_C \otimes \Delta^{n-1}, \Delta^1 := \Delta$ and subjecting the result to the $n$-fold multiplication map $B\otimes B \otimes \cdots \otimes B \rightarrow B$ of the algebra $B$.
 Hence one obtains the formula
 \begin{align*}
  &p(\psi_1 \star \psi_2 \star \cdots \star \psi_n (e_i)) \\
  \leq& \underbrace{\sum_{j_1,k_1} \sum_{j_2, k_2} \cdots \sum_{j_{n-1},k_{n-1}}}_{\# \text{of terms} = d^2 \cdot d^2 \cdots d^2 = (d^{2})^{n-1}} \underbrace{|\nu_{j_1 ,k_1}^{i}| |\nu_{j_2,k_2}^{k_1}| \cdots |\nu_{j_{n-1},k_{n-2}}^{k_{n-2}}|}_{\leq (\max_{i,j,k} \{|\nu_{i}^{jk}| ,1\})^{n-1} } p\left(\psi_{1} (e_{k_1})\prod_{2\leq r\leq n} \psi_{r} (e_{k_{r}})\right)    \\
  \leq& K^{n-1} M^n \leq (KM)^n.
 \end{align*}
 Again the estimate does neither depend on $e_i$ nor on the choice of $\psi_1, \ldots, \psi_n$, whence we see that one can take $KM$ in general as a constant which satisfies \eqref{eq: GN} for $q_\infty$ and $p_\infty$.
 We conclude that $A$ has the (GN)-property if $B$ has the (GN)-property.
 \item Let now $(B,\norm{\cdot})$ be a Banach algebra, then $A \cong B^d$ is a Banach space with the norm $\norm{\phi}_{\infty} := \max_{1 \leq i \leq d} \norm{\phi(e_i)}$.
 To prove that $A$ admits a submultiplicative norm, define the norm $\norm{\alpha}_K := K \norm{\alpha}_\infty$ (for $K$ the constant chosen above).
 By construction $\norm{\cdot}_K$ is equivalent to $\norm{\cdot}_\infty$ and we will prove that $\norm{\cdot}_K$ is submultiplicative.
 Consider $\alpha , \beta \in A$ and compute the norm of $\alpha \star \beta$ on a basis element
 \begin{align*}
  \norm{\alpha \star \beta (e_i)} &=\norm{m_B\circ (\alpha \otimes \beta)\circ \Delta (e_i)} \leq \sum_{j,k} |\nu_i^{jk}| \norm{\alpha(e_j)\beta(e_k)} \\
				  &\leq  \sum_{j,k} |\nu_i^{jk}| \norm{\alpha(e_j)}\norm{\beta(e_k)} \leq K \norm{\alpha}_\infty \norm{\beta}_\infty.
 \end{align*}
 In passing from the first to the second row we have used that the norm on $B$ is submultiplicative.
 Summing up, we obtain $\norm{\alpha \star \beta}_K \leq \norm{\alpha}_K \norm{\beta}_K$ whence $A$ is a Banach algebra.\qedhere
 \end{enumerate}
\end{proof}

In case the Hopf algeba $\Hopf$ is only of countable dimension, e.g.\ a Hopf algebra of finite type, and $B$ is a Banach algebra, the unit group $A^\times$ satisfies an even stronger regularity condition.

 \begin{lem}\label{lem: L1regular}
  Let $C = \bigoplus_{n \in \NN_0}$ be a graded coalgebra with $\mathrm{dim }\  C_0 < \infty$ and $B$ be a Banach algebra. Assume in addition that $C$ is of countable dimension, then the Lie group $A^\times$ from Proposition \ref{prop: unit group Lie group} is $L^1$-regular.
 \end{lem}

 \begin{proof}
  If $C$ is of countable dimension, then $A = \text{Hom}_\KK (C, B)$ is a \Frechet space (as it is isomorphic as a locally convex space to a countable product of Banach spaces).
  Now \cite[Proposition 7.5]{1601.02568v1} asserts that the unit group $A^\times$ of a continuous inverse algebra which is a \Frechet space will be $L^1$-regular if $A$ is locally $m$-convex.
  However, by the fundamental theorem for coalgebras \cite[Theorem 4.12]{MR2035107}, $C = \lim_{\rightarrow} \Sigma_n$ is the direct limit of finite dimensional subcoalgebras. Dualising, Lemma \ref{lem: GN-Banach0} shows that $A = \lim_{\leftarrow} \Hom_{\KK} (\Sigma_n, B)$ is the projective limit of Banach algebras. As Banach algebras are locally $m$-convex and the projective limit of a system of locally $m$-convex algebras is again locally $m$-convex (cf.\ e.g.\ \cite[Chapter III, Lemma 1.1]{MR857807}), the statement of the Lemma follows.
 \end{proof}

  To establish the following theorem we will show now that the Lie group $A^\times$ induces a Lie group structure on the character group of the Hopf algebra.
  Note that the character group is a closed subgroup of $A^\times$, but, contrary to the situation for finite dimensional Lie groups, closed subgroups do not automatically inherit a Lie group structure (see \cite[Remark IV.3.17]{MR2261066} for a counter example).
 \begin{thm}
 \label{thm: Char:Lie}
  Let $\Hopf = \bigoplus_{n\in \NN_0} \Hopf_n$ be a graded Hopf algebra with $\mathrm{dim}\ \Hopf_0 < \infty$.
  Then for any commutative CIA $B$, the group $\chG{\Hopf}{B}$ of $B$-valued characters of $\Hopf$ is a ($\KK-$analytic) Lie group.
  Furthermore, we observe the following.
  \begin{itemize}
    \item[\textup{(i)}] The Lie algebra $\Lf (\chG{\Hopf}{B})$ of $\chG{\Hopf}{B}$ is the Lie algebra $\chA{\Hopf}{B}$ of infinitesimal characters with the commutator bracket $\LB[\phi,\psi] = \phi \star \psi - \psi \star \phi$.
    \item[\textup{(ii)}] $\chG{\Hopf}{B}$ is a BCH--Lie group which is locally exponential, i.e.~the Lie group exponential map $\exp \colon \chA{\Hopf}{B} \rightarrow \chG{\Hopf}{B}, \ x \mapsto \sum_{n=0}^\infty \frac{x^{\star_n}}{n!}$ is a local $\KK$-analytic diffeomorphism.
    \end{itemize}
 \end{thm}

  \begin{proof}
   Recall from Proposition \ref{prop: unit group Lie group} and Proposition \ref{prop: topgp} that $\chG{\Hopf}{B}$ is a closed subgroup of the locally exponential Lie group $(A^\times, \star)$.
   We will now establish the Lie group structure using a well-known criterion for locally exponential Lie groups: Let $\exp_A$ be the Lie group exponential of $A^\times$ and consider the subset
   $$\Lf^e (\chG{\Hopf}{B}) := \{ x\in \Lf (A^\times) = A \mid \exp_{A} (\RR x) \subseteq \chG{\Hopf}{B}\}.$$
   We establish in Lemma \ref{lem: exp:bij} that $\chA{\Hopf}{B}$ is mapped by $\exp_{A}$ to $\chG{\Hopf}{B}$, whence $\chA{\Hopf}{B} \subseteq \Lf^e (\chG{\Hopf}{B})$.
   To see that $\Lf^e (\chG{\Hopf}{B})$ only contains infinitesimal characters, recall from Lemma \ref{lem: exp:bij} that there is an open $0$-neighborhood $\Omega \subseteq A$ such that $\exp_{A}$ maps $\chA{\Hopf}{B} \cap \Omega$ bijectively to $\exp_{A} (\Omega) \cap \chG{\Hopf}{B}$.
   If $x \in \Lf^e (\chG{\Hopf}{B}$ then we can pick $t > 0$ so small that $tx \in \Omega$. By definition of $\Lf^e (\chG{\Hopf}{B})$ we see that then $\exp_{A} (tx) \in \chG{\Hopf}{B} \cap \exp_{A} (\Omega)$.
   Therefore, we must have $tx \in \Omega \cap \chA{\Hopf}{B}$, whence $x \in \chA{\Hopf}{B}$.
   This entails that $\Lf^e (\chG{\Hopf}{B}) = \chA{\Hopf}{B}$ and then \cite[Theorem IV.3.3]{MR2261066} implies that $\chG{\Hopf}{B}$ is a locally exponential closed Lie subgroup of $(A^\times, \star)$ whose Lie algebra is the Lie algebra of infinitesimal characters $\chA{\Hopf}{B}$.
   Moreover, since $(A^\times ,\star)$ is a BCH--Lie group, so is $\chG{H}{B}$ (cf.\ \cite[Definition IV.1.9]{MR2261066}).
  \end{proof}

 Note that the Lie group $\chG{\Hopf}{B}$ constructed in Theorem \ref{thm: Char:Lie} will be modelled on a \Frechet space if $\Hopf$ is of countable dimension (e.g.\ if $\Hopf$ is of finite type) and $B$ is in addition a \Frechet algebra.
 If $\Hopf$ is even finite-dimensional and $B$ is a Banach algebra, then $\chG{\Hopf}{B}$ will even be modelled on a Banach space.

 \begin{ex}
   \begin{enumerate}
    \item The characters of a Hopf algebra of finite-type, i.e.\ the components $\Hopf_n$ in the grading $\Hopf = \oplus_{n\in\NN_0} \Hopf_n$ are finite-dimensional, are infinite-dimensional Lie groups by Theorem \ref{thm: Char:Lie}. Most natural examples of Hopf algebras appearing in combinatorial contexts are of finite-type.
    \item Every finite-dimensional Hopf algebra $\Hopf$ can be endowed with the trivial grading $\Hopf_0 := \Hopf$. Thus Theorem \ref{thm: Char:Lie} implies that character groups (with values in a commutative CIA) of finite-dimensional Hopf algebras (cf.\ \cite{MR2540955} for a survey) are infinite-dimensional Lie groups.
    \item Graded and connected Hopf algebras (see next section) appear in the Connes-Kreimer theory of perturbative renormalisation of quantum field theories.
    However, recently in \cite{MR3395225} it has been argued that instead of the graded and connected Hopf algebra of Feynman graphs considered traditionally (see e.g.\ the exposition in \cite{MR2371808}) a non connected extension of this Hopf algebra should be used.
    The generalisation of the Hopf algebra then turns out to be a Hopf algebra with $\text{dim } \Hopf_0 < \infty$, whence its character groups with values in a Banach algebra turn out to be infinite-dimensional Lie groups.
   \end{enumerate}
  These classes of examples could in general not be treated by the methods developed in \cite{BDS16}.
 \end{ex}

 \begin{rem}
  Recall that by definition of a graded Hopf algebra $\Hopf = \bigoplus_{n \in \NN_0} \Hopf_n$, the Hopf algebra structure turns $\Hopf_0$ into a sub-Hopf algebra. Therefore, we always obtain two Lie groups $\chG{\Hopf}{B}$ and $\chG{\Hopf_0}{B}$ if $\text{dim } \Hopf_0 < \infty$.
  It is easy to see that the restriction map $q \colon \chG{\Hopf}{B} \rightarrow \chG{\Hopf_0}{B} , \phi \mapsto \phi|_{\Hopf_0}$ is a morphism of Lie groups with respect to the Lie group structures from Theorem \ref{thm: Char:Lie}.
  Its kernel is the normal subgroup $$\text{ker} q = \{ \phi \in \chG{\Hopf}{B} \mid \phi|_{\Hopf_0} = \one|_{\Hopf_0}\},$$
  i.e.\ the group of characters which coincide with the unit on degree $0$.
 \end{rem}

 We now turn to the question whether the Lie group constructed in Theorem \ref{thm: Char:Lie} is a \emph{regular} Lie group.

  \begin{thm}													\label{thm: Char_regular}
   Let $\Hopf$ be a graded Hopf algebra with $\text{dim }\Hopf_0 < \infty$ and $B$ be a Banach algebra.
   Then the Lie group $\chG{\Hopf}{B}$ is $C^0$-regular and the associated evolution map is even a $\KK$-analytic map.
   \end{thm}

  \begin{proof}
   In Theorem \ref{thm: Char:Lie} the Lie group structure on the character group was identified as a closed subgroup of the Lie group $A^\times$.
   By definition of the character group (cf.\ Definition \eqref{defn: char}), $\chG{\Hopf}{B}$ can be described as
   \begin{displaymath}
    \chG{\Hopf}{B} = \{\phi \in A^\times \mid \phi \circ m_\Hopf = m_B \circ (\phi \otimes \phi) \}
   \end{displaymath}
   As discussed in Remark \ref{setup: dual_comult} Equation \eqref{eq: precomp:mH}, the map $(m_\Hopf)^* \colon A^\times \rightarrow A_{\otimes}^\times, \phi \mapsto \phi \circ m_\Hopf$ is a Lie group morphism.
   Now we consider the map $\theta \colon A \rightarrow A_{\otimes} , \phi \mapsto m_B \circ (\phi \otimes \phi)$.
   Observe that $\theta = \beta \circ \text{diag}$, where $\beta$ is the continuous bilinear map \eqref{eq: beta} and $\text{diag} \colon A \rightarrow A \times A , \phi \mapsto (\phi , \phi)$.
   Since $\beta$ is smooth (as a continuous bilinear map), and the diagonal map is clearly smooth, $\theta$ is smooth.
   Further, \eqref{eq: multifalt} shows that $\theta (\phi \star \psi) = \theta(\phi) \star_{At} \theta(\psi)$.
   Thus $\theta$ restricts to a Lie group morphism $\theta_{A^\times} \colon A^\times \rightarrow At^\times$.

   Summing up, we see that $\chG{\Hopf}{B} = \{ \phi \in A^\times \mid (m_\Hopf)^*(\phi) = \theta_{A^\times}(\phi) \}$.
   Since by Proposition \ref{prop: unit group Lie group} the Lie group $A^\times$ is $C^0$-regular, the closed subgroup $\chG{\Hopf}{B}$ is also $C^0$-regular by \cite[Theorem G]{1208.0715v3}.
  \end{proof}

 \begin{rem}\label{rem: L1reg}
  In Theorem \ref{thm: Char_regular} we have established that the character group $\chG{\Hopf}{B}$ inherits the regularity properties of the ambient group of units.
  If the Hopf algebra $\Hopf$ is in addition of countable dimension, e.g.\  a Hopf algebra of finite type, then Lemma \ref{lem: L1regular} asserts that the ambient group of units is even $L^1$-regular.
  Unfortunately, Theorem \cite[Theorem G]{1208.0715v3} only deals with regularity of type $C^k$ for $k \in \NN_0 \cup \{\infty\}$.
  However, since $\chG{\Hopf}{B}$ is a closed Lie subgroup of $\Hom_\KK (\Hopf,B)^\times$, it is easy to see that the proof of \cite[Theorem G]{1208.0715v3} carries over without any changes to the $L^1$-case.\footnote{One only has to observe that a function into the Lie subgroup is absolutely continuous if and only if it is absolutely continuous as a function into the larger group. On the author's request, a suitable version of \cite[Theorem G]{1208.0715v3} for $L^1$-regularity will be made available in a future version of \cite{1601.02568v1}.}
  Hence, we can adapt the proof of Theorem \ref{thm: Char_regular} to obtain the following statement:
  \smallskip

  \textbf{Corollary} \emph{Let $\Hopf$ be a graded Hopf algebra of countable dimension with $\text{dim }\Hopf_0 < \infty$ and $B$ be a Banach algebra.
   Then the Lie group $\chG{\Hopf}{B}$ is $L^1$-regular.}
 \end{rem}

  Note that the results on $L^1$-regularity of character groups considerably strengthen the results which have been available for regularity of character groups (see \cite{BDS16}).

\subsection{Character groups for a graded and connected Hopf algebra $\Hopf$ and $B$ a locally convex algebra}\label{sect: conn}

In many interesting cases, the Hopf algebra is even connected graded (i.e. $\Hopf_0$ is one-dimensional). In this case, we can weaken the assumption on $B$ to be only locally convex.

 \begin{thm}[{\cite[Theorem 2.7]{BDS16}}]\label{thm: main}
 Let $\Hopf$ be a graded and connected Hopf algebra and $B$ be a locally convex algebra.
 Then the manifold structure induced by the global parametrisation $\exp \colon \chA{\Hopf}{B} \rightarrow \chG{\Hopf}{B},\ x \mapsto \sum_{n \in \NN_0} \frac{x^{\star_n}}{n!}$ turns $\chG{\Hopf}{B}$ into a $\KK$-analytic Lie group.

 The Lie algebra associated to this Lie group is $\chA{\Hopf}{B}$.
 Moreover, the Lie group exponential map is given by the real analytic diffeomorphism $\exp$. 
 \end{thm}

 \begin{rem}
 Note that Theorem \ref{thm: main} yields more information for the graded and connected case then the new results discussed in Section \ref{sect: noncon}: In the graded and connected case, the Lie group exponential is a global diffeomorphism, whereas the theorem for the non-connected case only establishes that $\exp$ induces a diffeomorphism around the unit.
 \end{rem}

 Note that the connectedness of the Hopf algebra is the important requirement here.
 Indeed, we can drop the assumption of an integer grading and generalise to the grading by a suitable monoid.

 \begin{defn}
  Let $M$ be a submonoid of $(\RR, +)$ with $0 \in M$.
  We call $M$ an \emph{index monoid} if every initial segment $I_m := \{n \in M \mid n \leq m\}$ is finite.
  Note that this forces the monoid to be at most countable.

  As in the $\NN_0$ graded case, we say a Hopf algebra $\Hopf$ is \emph{graded by an index monoid} $M$, if $\Hopf = \bigoplus_{m \in M} \Hopf_m$ and the algebra, coalgebra and antipode respect the grading in the usual sense.
 \end{defn}

 \begin{ex}
  The monoid $(\NN_0,+)$ is an index monoid.

  A source for (more interesting) index monoids is Hairer's theory of regularity structures for locally subcritical semilinear stochastic partial differential equations \cite[Section 8.1]{MR3274562} (cf.\ in particular \cite[Lemma 8.10]{MR3274562} and see Example \ref{ex: Hairer} below).
 \end{ex}

   Note that the crucial property of an index monoid is the finiteness of initial segments. This property allows one to define the functional calculus used in the proof of \cite[Theorem B]{BDS16} to this slightly more general setting.
   Changing only trivial details in the proof of loc.cit., one immediately deduces that the statement of Theorem \ref{thm: main} generalises from an $\NN_0$-grading to the more general situation
   \begin{cor}\label{cor: igraded}
    Let $\Hopf = \bigoplus_{m \in M} \Hopf_m$ be a graded and connected Hopf algebra graded by an index monoid $M$, $B$ a sequentially complete locally convex algebra.
    Then the manifold structure induced by $\exp \colon \chA{\Hopf}{B} \rightarrow \chG{\Hopf}{B},\ x \mapsto \sum_{n \in \NN_0} \frac{x^{\star_n}}{n!}$ turns $\chG{\Hopf}{B}$ into a $\KK$-analytic Lie group.
    This Lie group is $C^0$-regular, its Lie algebra is $\chA{\Hopf}{B}$ and the Lie group exponential map is the real analytic diffeomorphism $\exp$.
   \end{cor}

 \subsubsection{Application to character groups in the renormalisation of SPDEs}

  Hopf algebras graded by index monoids appear in Hairer's theory of regularity structures for locally subcritical semilinear SPDEs \cite[Section 8.1]{MR3274562}.
  Character groups of these Hopf algebras (and their quotients) then appear as structure group in Hairer's theory (cf.\ \cite[Theorem 8.24]{MR3274562} and \cite{1610.08468v1}) of regularity structures.
  Recall that a regularity structure $(A,T,G)$ in the sense of Hairer (cf.\ \cite[Definition 2.1]{MR3274562}) is given by
  \begin{itemize}
   \item an index set $A \subseteq \RR$ with $0 \in A$, which is bounded from below and locally finite,
   \item a \emph{model space} $T = \bigoplus_{\alpha \in A} T_\alpha$ which is a graded vector space with $T_0 \cong \RR$ (denote by $1$ its unit element) and $T_\alpha$ a Banach space for every $\alpha \in A$,
   \item a \emph{structure group} $G$ of linear operators acting on $T$ such that, for every $\Gamma \in G$, $\alpha \in A$ and $a \in T_\alpha$ one has
   $$\Gamma a - a \in \bigoplus_{\beta < \alpha} T_\beta \text{ and } \Gamma 1 =1.$$
  \end{itemize}
  We sketch now briefly the theory developed in \cite{1610.08468v1}, where for a class of examples singular SPDEs the structure group was recovered as the character group of a connected Hopf algebra graded by an index monoid.

  \begin{ex}\label{ex: Hairer}
   The construction outlined in \cite{1610.08468v1} builds first a general Hopf algebra of decorated trees in the category of bigraded spaces.   Note that this bigrading does not induce a suitable $\NN_0$-grading (or grading by index monoid) for our purposes.
   This Hopf algebra encodes the combinatorics of Taylor expansions in the SPDE setting and it needs to be tailored to the specific SPDE. This is achieved by choosing another ingredient, a so called \emph{subcritical and complete normal rule}, i.e.\ a device depending on the SDE in question which selects a certain sub-Hopf algebra (see \cite[Section 5]{1610.08468v1} for details).
   Basically, the rule collects all admissible terms (= abstract decorated trees) which appear in the local Taylor expansion of the singular SPDE.\footnote{see \cite[Section 5.5]{1610.08468v1} for some explicit examples of this procedure, e.g.\ for the KPZ equation.}

   Using the rule, we can select an algebra of decorated trees $\mathcal{T}^{ex}_{+}$ admissible with respect to the rule. Here the ``$+$'' denotes that we only select trees which are positive with respect to a certain grading $|\cdot|_+$  (cf.\ \cite[Remark 5.3 and Definition 5.29]{1610.08468v1}).
   Then \cite[Proposition 5.34]{1610.08468v1} shows that $(\mathcal{T}^{ex}_{+}, |\cdot|_+)$ is a graded and connected Hopf algebra of decorated trees.
   Note however, that the grading $|\cdot|_+$ is in general not integer valued, i.e.\ $\mathcal{T}^{ex}_{+}$ is graded by a submonoid $M$ of $[0,\infty[$.

   Since we are working with a normal rule which is complete and subcritical, the submonoid $M$ satisfies $|\{ \gamma \in M \mid \gamma < c\}| < \infty$ for each $c \in \RR$, i.e.\ $M$ is an index monoid. The reason for this is that by construction $\mathcal{T}^{ex}_{+}$ is generated by tree products of trees which strongly conform to the rule \cite[Lemma 5.25 and Definition 5.29]{1610.08468v1}. As the rule is complete and subcritical, there are only finitely many trees $\tau$ with $|\tau|_+ < c $ (for $c \in \RR$) which strongly conform to the rule.
   Now the tree product is the Hopf algebra product (i.e.\ the product respects the grading), whence the poperty for $M$ follows.

   Now by \cite[Proposition 5.39]{1610.08468v1} the graded space $\mathcal{T}^{ex} = (\langle B_0\rangle , | \cdot |_+)$ (suitably generated by a certain subset of the strongly conforming trees) together with the index set $A^{ex} = \{|\tau|_+ \mid \tau \in B_0\}$ and the character group $\mathcal{G}_+^{ex} :=\chG{\mathcal{T}^{ex}_{+}}{\RR}$ form a regularity structure $(A^{ex}, \mathcal{T}^{ex}, \mathcal{G}_+^{ex})$. In conclusion, Corollary \ref{cor: igraded} yields an infinite-dimensional Lie group structure for the structure group $\mathcal{G}_+^{ex}$ of certain subcritical singular SPDEs.
   \end{ex}

  \begin{rem}
   In the full renormalisation procedure outlined in \cite{1610.08468v1} two groups are involved in the renormalisation. Apart from the structure group outlined above, also a so called \emph{renormalisation group} $\mathcal{G}_{-}^\text{ex}$ is used (cf.\ \cite[Section 5]{1610.08468v1}).
   This group $\mathcal{G}_{-}^{ex}$ is (in the locally subcritical case) a finite-dimensional group arising as the character group of another Hopf algebra. However, it turns out that the actions of $\mathcal{G}_{+}^{ex}$ and $\mathcal{G}_{-}^{ex}$ interact in an interesting way (induced by a cointeraction of underlying Hopf algebras, think semidirect product). We hope to return to these actions and use the (infinite-dimensional) Lie group structure to study them in future work.
  \end{rem}

\begin{acknowledgement}
This research was partially supported by the European Unions Horizon 2020 research and innovation
programme under the Marie Sk\l{}odowska-Curie grant agreement No.\ 691070 and by the Knut and Alice Wallenberg Foundation grant agreement KAW~2014.0354.
We are indebted to K.--H.\ Neeb and R.\ Dahmen for discussions which led to Lemma \ref{lem: cia0}.
Further, we would like to thank L.\ Zambotti and Y.\ Bruned for explaining their
results on character groups in the renormalisation of SPDEs.
Finally, we thank K.H.\ Hofmann for encouraging and useful comments and
apologize to him for leaving out \cite{MR3114697} at first.
\end{acknowledgement}

\section{Appendix: Infinite-dimensional calculus}\label{app: calculus}
  In this section basic facts on the differential calculus in infinite-dimensional spaces are recalled.
 The general setting for our calculus are locally convex spaces (see \cite{MR0342978, MR632257}).

 \begin{defn}
  Let $E$ be a topological vector space over $\KK \in \{\RR,\CC \}$.
    $E$ is called \emph{locally convex space} if there is a family $\{p_i \mid i \in I\}$ of continuous seminorms for some index set $I$, such that
      \begin{enumerate}
      \item[i.] the topology is the initial topology with respect to
$\{\text{pr}_{p_i} \colon E \rightarrow E_{p_i} \mid i \in I\}$, i.e.\ the
$E$-valued map $f$ is continuous if and only if $\text{pr}_i \circ f$ is
continuous for each $i\in I$, where $E_{p_i} := E/p_i^{-1} (0)$ is the
normed space associated to the $p_i$ and $\text{pr}_i \colon E \rightarrow E_{p_i}$ is the canonical projection,
      \item[ii.] if $x \in E$ with $p_i (x) = 0$ for all $i \in I$, then $x = 0$ (i.e.\ $E$ is Hausdorff).
     \end{enumerate}
 \end{defn}

 Many familiar results from finite-dimensional calculus carry over to infinite dimensions if we assume that all spaces are locally convex.

  As we are working beyond the realm of Banach spaces, the usual notion of Fr\'{e}chet-differentiability cannot be used.\footnote{The problem here is that the bounded linear operators do not admit a good topological structure if the spaces are not normable. In  particular, the chain rule will not hold for Fr\'{e}chet-differentiability in general for these spaces (cf.\ \cite{keller1974}).}
  Moreover, there are several inequivalent notions of differentiability on locally convex spaces (see again \cite{keller1974}).
  
  We base our investigation on the so called Bastiani calculus, \cite{MR0177277}.
  The notion of differentiability we adopt is natural and quite simple, as the derivative is defined via directional derivatives.

\begin{defn}\label{defn: deriv}
 Let $\KK \in \{\RR,\CC\}$, $r \in \NN \cup \{\infty\}$ and $E$, $F$ locally convex $\KK$-vector spaces and $U \subseteq E$ open.
 Moreover we let $f \colon U \rightarrow F$ be a map.
 If it exists, we define for $(x,h) \in U \times E$ the directional derivative
 $$df(x,h) := D_h f(x) := \lim_{\KK^\times \ni t\rightarrow 0} t^{-1} (f(x+th) -f(x)).$$
 We say that $f$ is $C^r_\KK$ if the iterated directional derivatives
    \begin{displaymath}
     d^{(k)}f (x,y_1,\ldots , y_k) := (D_{y_k} D_{y_{k-1}} \cdots D_{y_1} f) (x)
    \end{displaymath}
 exist for all $k \in \NN_0$ such that $k \leq r$, $x \in U$ and $y_1,\ldots , y_k \in E$ and define continuous maps $d^{(k)} f \colon U \times E^k \rightarrow F$.
 If it is clear which $\KK$ is meant, we simply write $C^r$ for $C^r_\KK$.
 If $f$ is $C^\infty_\CC$, we say that $f$ is \emph{holomorphic} and if $f$ is $C^\infty_\RR$ we say that $f$ is \emph{smooth}.
\end{defn}

For more information on our setting of differential calculus we refer the reader to \cite{MR1911979,keller1974}.
Another popular choice for infinite-dimensional calculus is the so called \textquotedblleft convenient setting\textquotedblright\ of global analysis outlined in \cite{KM97}.
On Fr\'{e}chet spaces (i.e.\ complete metrisable locally convex spaces) our notion of differentiability coincides with differentiability in the sense of convenient calculus. 
Note that differentiable maps in our setting are continuous by default (which is in general not true in the convenient setting).
We encounter analytic mappings between infinite-dimensional spaces, as a preparation for this, note first:

\begin{rem}
 A map $f \colon U \rightarrow F$ is of class $C^\infty_\CC$ if and only if it is \emph{complex analytic} i.e.,
  if $f$ is continuous and locally given by a series of continuous homogeneous polynomials (cf.\ \cite[Proposition 7.4 and 7.7]{MR2069671}).
  We then also say that $f$ is of class $C^\omega_\CC$.
\label{rem: analytic}
\end{rem}

To introduce real analyticity, we have to generalise a suitable characterisation from the finite-dimensional case:
A map $\RR \rightarrow \RR$ is real analytic if it extends to a complex analytic map $\CC \supseteq U \rightarrow \CC$ on an open $\RR$-neighbourhood $U$ in $\CC$.
We proceed analogously for locally convex spaces by replacing $\CC$ with a suitable complexification.

\begin{defn}[Complexification of a locally convex space]
 Let $E$ be a real locally convex topological vector space.
 Endow $E_\CC := E \times E$ with the following operation
 \begin{displaymath}
  (x+iy).(u,v) := (xu-yv, xv+yu) \quad \mbox{ for } x,y \in \RR, u,v \in E.
 \end{displaymath}
 The complex vector space $E_\CC$ with the product topology is called the \emph{complexification} of $E$. We identify $E$ with the closed real subspace $E\times \{0\}$ of $E_\CC$.
\end{defn}

\begin{defn}
 Let $E$, $F$ be real locally convex spaces and $f \colon U \rightarrow F$ defined on an open subset $U$.
 Following \cite{MR830252} and \cite{MR1911979}, we call $f$ \emph{real analytic} (or $C^\omega_\RR$) if $f$ extends to a $C^\infty_\CC$-map $\tilde{f}\colon \tilde{U} \rightarrow F_\CC$ on an open neighbourhood $\tilde{U}$ of $U$ in the complexification $E_\CC$.\footnote{If $E$ and $F$ are Fr\'{e}chet spaces, real analytic maps in the sense just defined coincide with maps which are continuous and can be locally expanded into a power series. See \cite[Proposition 4.1]{MR2402519}.}
\end{defn}

Note that many of the usual results of differential calculus carry over to our setting.
In particular, maps on connected domains whose derivative vanishes are constant as a version of the fundamental theorem of calculus holds.
Moreover, the chain rule holds in the following form:

\begin{lem}[{Chain Rule {\cite[Propositions 1.12, 1.15, 2.7 and 2.9]{MR1911979}}}]
 Fix $k \in \NN_0 \cup \{\infty , \omega\}$ and $\KK \in \{\RR , \CC\}$ together with $C^k_\KK$-maps $f \colon E \supseteq U \rightarrow F$ and $g \colon H \supseteq V \rightarrow E$ defined on open subsets of locally convex spaces.
 Assume that $g (V) \subseteq U$. Then $f\circ g$ is of class $C^{k}_\KK$ and the first derivative of $f\circ g$ is given by
  \begin{displaymath}
   d(f\circ g) (x;v) = df(g(x);dg(x,v)) \quad  \mbox{for all } x \in U ,\ v \in H.
  \end{displaymath}
\end{lem}

The differential calculus developed so far extends easily to maps which are defined on non-open sets.
This situation occurs frequently in the context of differential equations on closed intervals (see \cite{alas2012} for an overview).

Having the chain rule at our disposal we can define manifolds and related constructions which are modelled on locally convex spaces.

\begin{defn} Fix a Hausdorff topological space $M$ and a locally convex space $E$ over $\KK \in \{\RR,\CC\}$.
An ($E$-)manifold chart $(U_\kappa, \kappa)$ on $M$ is an open set $U_\kappa \subseteq M$ together with a homeomorphism $\kappa \colon U_\kappa \rightarrow V_\kappa \subseteq E$ onto an open subset of $E$.
Two such charts are called $C^r$-compatible for $r \in \NN_0 \cup \{\infty,\omega\}$ if the change of charts map $\nu^{-1} \circ \kappa \colon \kappa (U_\kappa \cap U_\nu) \rightarrow \nu (U_\kappa \cap U_\nu)$ is a $C^r$-diffeomorphism.
A $C^r$-atlas of $M$ is a set of pairwise $C^r$-compatible manifold charts, whose domains cover $M$. Two such $C^r$-atlases are equivalent if their union is again a $C^r$-atlas.

A \emph{locally convex $C^r$-manifold} $M$ modelled on $E$ is a Hausdorff space $M$ with an equivalence class of $C^r$-atlases of ($E$-)manifold charts.
\end{defn}

 Direct products of locally convex manifolds, tangent spaces and tangent bundles as well as $C^r$-maps of manifolds may be defined as in the finite dimensional setting (see \cite[I.3]{MR2261066}).
 The advantage of this construction is that we can now give a very simple answer to the question, what an infinite-dimensional Lie group is:

\begin{defn}
A (locally convex) \emph{Lie group} is a group $G$ equipped with a $C_\KK^\infty$-manifold structure modelled on a locally convex space, such that the group operations are smooth.
If the manifold structure and the group operations are in addition ($\KK$-) analytic, then $G$ is called a ($\KK$-) \emph{analytic} Lie group.
\end{defn}

We recommend \cite{MR2261066} for a survey on the theory of locally convex Lie groups.
However, the Lie groups constructed in this article have strong structural properties as they belong to the class of Baker--Campbell--Hausdorff--Lie groups.
\begin{defn}[Baker--Campbell--Hausdorff (BCH-)Lie groups and Lie algebras] \label{defn: BCH}\mbox{}
\begin{enumerate}
 \item A Lie algebra $\mathfrak{g}$ is called
\emph{Baker--Campbell--Hausdorff--Lie algebra} (BCH--Lie algebra) if there
exists an open $0$-neighbourhood $U \subseteq \mathfrak{g}$ such that for $x, y
\in U$ the \emph{BCH-series} $\sum_{n=1}^\infty H_n (x,y)$ converges and
defines an analytic map $U \times U \rightarrow \mathfrak{g}$.
The $H_n$ are defined as $H_1 (x,y) = x +y$, $H_2 (x,y) = \frac{1}{2}\LB[x,y]$
and for $n\geq 3$ by linear combinations of iterated brackets, see
\cite[Definition IV.1.5.]{MR2261066} or \cite[Chapter 2, \S 6]{MR1728312}.
  \item A locally convex Lie group $G$ is called \emph{BCH--Lie group} if it satisfies one of the following equivalent conditions (cf.\  \cite[Theorem IV.1.8]{MR2261066})
    \begin{enumerate}[\textup (i)]
     \item $G$ is a $\KK$-analytic Lie group whose Lie group exponential function is $\KK$-analytic and a local diffeomorphism in $0$.
     \item The exponential map of $G$ is a local diffeomorphism in $0$ and $\Lf (G)$ is a BCH--Lie algebra.
    \end{enumerate}
\end{enumerate}
\end{defn}

BCH--Lie groups share many of the structural properties of Banach Lie groups while not necessarily being Banach Lie groups themselves.

 \section{Appendix: Characters and the exponential map}
  Fix for the rest of this section a $\KK$-Hopf algebra $\Hopf=(\Hopf,m_\Hopf,u_\Hopf,\Delta_\Hopf,\varepsilon_\Hopf,S_\Hopf)$ and a commutative continuous inverse algebra $B$.
  Furthermore, we assume that the Hopf algebra $\Hopf$ is graded, i.e.\ $\Hopf = \bigoplus_{n \in \NN_0} \Hopf$ and $\text{dim }\Hopf_0 < \infty$.
  The aim of this section is to prove that the Lie group exponential map $\exp_{A^\times}$ of $A := \Hom_\KK (\Hopf, B)$ restricts to a bijection from the infinitesimal characters to the characters.

 \begin{rem}[Cocomposition with Hopf multiplication]											\label{setup: dual_comult}
  Let $\Hopf \otimes \Hopf$ be the tensor Hopf algebra (cf.\ \cite[p. 8]{MR1381692}).
  We regard the tensor product $\Hopf\otimes \Hopf$ as a graded and connected Hopf algebra with the tensor grading, i.e.\ $\Hopf \otimes \Hopf = \bigoplus_{n \in \NN_0} (\Hopf \otimes \Hopf)_n$ where for all $n \in \NN_0$ the $n$th degree is defined as $(\Hopf \otimes \Hopf)_n = \bigoplus_{i+j = n} \Hopf_i \otimes \Hopf_j$.

  Since $\text{dim }\Hopf_0 < \infty$ we see that $\text{dim }(\Hopf \otimes \Hopf)_0 = \text{dim } \Hopf_0 \otimes \Hopf_0 < \infty$
  Thus with respect to the topology of pointwise convergence and the convolution product, the algebras
   \begin{align*}
    A := \Hom_\KK (\Hopf,B) \quad \quad A_{\otimes} := \Hom_\KK (\Hopf \otimes \Hopf , B)
   \end{align*}
   become continuous inverse algebras (see \ref{setup: graded} and Lemma \ref{lem: cia0}).
  This structure turns
  \begin{displaymath}
   \cdot \circ m_\Hopf \colon \Hom_\KK(\Hopf,B \rightarrow \Hom_\KK(\Hopf\otimes \Hopf,B), \quad \phi \mapsto \phi\circ m_\Hopf.
  \end{displaymath}
  into a continuous algebra homomorphism. Hence its restriction
  \begin{equation}\label{eq: precomp:mH}
   (m_\Hopf)^* \colon A^\times \rightarrow A_{\otimes}^\times , \quad \phi \mapsto \phi \circ m_\Hopf
  \end{equation}
 is a Lie group morphism with $\Lf ((m_\Hopf)^*) := T_e (m_\Hopf)^* = \cdot \circ m_\Hopf$.
 \end{rem}

 \begin{lem}\label{lem: exp:bij}
  The Lie group exponential $\exp_{A} \colon \Lf (A^\times) = A \rightarrow A^\times,\ x \mapsto \sum_{n=0}^\infty \frac{x^{\star_n}}{n!}$ maps $\chA{\Hopf}{B}$ to $\chG{\Hopf}{B}$.
  Further, there is a $0$-neighborhood $\Omega \subseteq A$ such that $\exp_{A}$ maps $\chA{\Hopf}{B} \cap \Omega$ bijectively onto $\chG{\Hopf}{B} \cap \exp_{A} (\Omega)$.
  \footnote{Note that apart from the locallity and several key arguments, the proof follows the general idea of the similar statement \cite[Lemma B.10]{BDS16}. For the readers convenience we repeat the arguments to exhibit how properties of the Lie group exponential replace the functional calculus used in \cite{BDS16}.}
 \end{lem}

 \begin{proof}
  We denote by $\star_{A_\otimes}$ the convolution product of the CIA $A_{\otimes}$ and let $\exp_{A_{\otimes}}$ be the Lie group exponential of $A_{\otimes}^\times$.
  Let $m_B \colon B\otimes B \rightarrow B, \ b_1\otimes b_2 \mapsto b_1\cdot b_2$ be the multiplication in $B$.
  Define the continuous bilinear map (cf.\ \cite[proof of Lemma B.10]{BDS16} for detailed arguments)
  \begin{equation}\label{eq: beta}
   \beta \colon A \times A \rightarrow A_{\otimes}, \quad (\phi,\psi) \mapsto \phi \td \psi :=  m_B \circ (\phi \otimes \psi).
  \end{equation}
  We may use $\beta$ to rewrite the convolution in $A$ as $\star_A = \beta \circ \Delta$ and obtain
  \begin{equation}\label{eq: multifalt}
   (\phi_1 \td \psi_1 )\star_{A_\otimes} (\phi_2 \td \psi_2) = (\phi_1\star_A \phi_2) \td (\psi_1 \star_A \psi_2).
  \end{equation}
  Recall, that $1_A:= u_B\circ\varepsilon_\Hopf$ is the neutral element of the algebra $A$.
  From equation \eqref{eq: multifalt}, it follows at once, that the continuous linear maps
  \begin{displaymath}
   \beta (\cdot, 1_A)  \colon A \rightarrow A_{\otimes}, \quad \phi \mapsto \phi \td 1_A \quad \text{ and } \quad \beta (1_A, \cdot) \colon A \rightarrow A_{\otimes}, \quad \phi \mapsto 1_A \td \phi
  \end{displaymath}
  are algebra homomorphisms which restrict to Lie group morphisms
   \begin{equation}\label{eq:LG:hom}
    \beta^\rho \colon A^\times \rightarrow A_{\otimes}^\times ,\quad \phi \mapsto \beta (\phi,1_A) \text{ and } \beta^\lambda \colon A^\times \rightarrow A_{\otimes}^\times ,\quad \phi \mapsto \beta (1_A,\phi)
   \end{equation}
  with $\Lf (\beta^\rho) =  \beta (\cdot, 1_A)$ and $\Lf (\beta^\lambda) =  \beta (1_A, \cdot)$.
  Let $\phi\in A$ be given and recall from \eqref{eq: multifalt} that $(\phi \td 1_A) \star_{A_\otimes} (1_A \td \phi) = \phi \td \phi =  (1_A \td \phi) \star_{A_\otimes}  (\phi \td 1_A)$.
  As a consequence we obtain
      \begin{equation}\label{eq: com:ptm}
       \exp_{A_{\otimes}} (\phi \td 1_A + 1_A \td \phi) = \exp_{A_{\otimes}} (\phi \td 1_A) \star_{A_\otimes} \exp_{A_{\otimes}} (1_A \td \phi).
      \end{equation}
  since every Lie group exponential function transforms addition into multiplication for commuting elements.

  Note that it suffices to check multiplicativity of $\exp_A (\phi)$ as $\exp_{A}(\phi)(1_\Hopf)=1_B$ is automatically satisfied.
  For an infinitesimal character $\phi \in A$ we have by definition $\phi \circ m_\Hopf = \phi \td 1_A + 1_A \td \phi$. Using now the naturality of the Lie group exponentials (i.e.\ for a Lie group morphism $f \colon G \rightarrow H$ we have $\exp_H \circ \Lf (f) = f\circ \exp_G$), we derive the following: \allowdisplaybreaks
  \begin{align*}
	\phi 	\in \chA{\Hopf}{B} 	& \stackrel{\text{Def}}{\iff} \phi \circ m_\Hopf = \phi \td 1_A + 1_A \td \phi\\
						& \stackrel{\, (!)}{\ \, \Longrightarrow}     \exp_{A_{\otimes}}(\phi \circ m_\Hopf)  	 = 	\exp_{A_{\otimes}}(\phi \td 1_A + 1_A \td \phi)						\\
							& \stackrel{\eqref{eq: com:ptm}}{\iff} \exp_{A_{\otimes}}(\phi \circ m_\Hopf)  	 = 	\exp_{A_\otimes}(\phi \td 1_A) \star_{A_\otimes} \exp_{A_\otimes}(1_A \td \phi)		\\
							&\stackrel{\eqref{eq:LG:hom}}{\iff} \exp_{A_{\otimes}}(\phi \circ m_\Hopf)  	 = 	\bigl( \exp_{A}(\phi) \td 1_A  \bigr)\star_{A_\otimes} \bigl(1_A \td \exp_{A}(\phi) \bigr) 		\\
							&\stackrel{\eqref{eq: multifalt}}{\iff} \exp_{A_{\otimes}}(\phi \circ m_\Hopf)  	 = 	\bigl( \exp_{A}(\phi)\star_{A} 1_A \bigr) \td \bigl(1_A  \star_{A} \exp_{A}(\phi) \bigr)	 \\
							&\iff \exp_{A_{\otimes}}(\phi \circ m_\Hopf)  	 = 	\exp_{A}(\phi) \td  \exp_{A}(\phi) 							\\
						&\stackrel{\eqref{eq: precomp:mH}}{\iff}
						      \exp_{A}(\phi) \circ m_\Hopf  		 = 	\exp_{A}(\phi) \td  \exp_{A}(\phi)	\\
						&\stackrel{\text{Def}}{\iff}
						      \exp_{A}(\phi) \in \chG{\Hopf}{B} .
  \end{align*}
  This shows that infinitesimal characters are mapped by the Lie group exponential to elements in the character group.

  Now we observe that in general the implication from the first to the second line will not be an equivalence (as the Lie group exponential is not a global diffeomorphism unlike the connected Hopf algebra case discussed in \cite[Lemma B.10]{BDS16}).
  We exploit now that $A^\times$ and $A_{\otimes}^\times$ are locally exponential Lie groups, whence locally around $0$ in $A$ and $A_{\otimes}$ the Lie group exponentials induce diffeomorphisms.
  Hence there are open neighborhoods of $0$ and the units of $A^\times$ and $A_{\otimes}^\times$, such that $\exp_{A^\times} \colon A \supseteq V \rightarrow W \subseteq A^\times$ and $\exp_{A_{\otimes}^\times} \colon A_{\otimes} \supseteq V_\otimes \rightarrow W_{\otimes} \subseteq A_{\otimes}^\times$ are diffeomorphisms.
  Since $A_{\otimes}$ is a locally convex space, there is an open $0$-neighborhood $\Omega_\otimes \subseteq A_{\otimes}$ such that $\Omega_\otimes + \Omega_\otimes \subseteq V_\otimes$.
  By continuity of $\beta$ and $\cdot \otimes m_\Hopf$, we obtain an open open $0$-neighborhood
  $$\Omega := V \cap (\cdot \otimes m_\Hopf)^{-1} (V_\otimes) \cap \beta (\cdot , 1_A)^{-1} (\Omega_\otimes) \cap \beta (1_{A},\cdot)^{-1} (\Omega_\otimes) \subseteq A.$$
  Now by construction elements in $\chA{\Hopf}{B} \cap \Omega$ are mapped by $\cdot \circ m_\Hopf$ into $V_\times$ and by $\beta (\cdot , 1_A) + \beta (1_A,\cdot)$ into $\Omega_\otimes + \Omega_\otimes \subseteq V_\otimes$.
  Since $\exp_{A_{\otimes}}$ induces a diffeomorphism on $V_\otimes$, the implication $(!)$ becomes an equivalence for elements in $\chA{\Hopf}{B} \cap \Omega$. We have thus established that $\exp_{A}$ maps $\chA{\Hopf}{B} \cap \Omega$ bijectively to $\exp_{A}(\Omega) \cap \chG{\Hopf}{B}$.
\end{proof}

\bibliographystyle{spmpsci}
\bibliography{Abel} 

\newcommand{\noopsort}[1]{} \newcommand{\singleletter}[1]{#1}
  \def\polhk#1{\setbox0=\hbox{#1}{\ooalign{\hidewidth
  \lower1.5ex\hbox{`}\hidewidth\crcr\unhbox0}}}
\begin{thebibliography}{10}
\providecommand{\url}[1]{{#1}}
\providecommand{\urlprefix}{URL }
\expandafter\ifx\csname urlstyle\endcsname\relax
  \providecommand{\doi}[1]{DOI~\discretionary{}{}{}#1}\else
  \providecommand{\doi}{DOI~\discretionary{}{}{}\begingroup
  \urlstyle{rm}\Url}\fi

\bibitem{alas2012}
Alzaareer, H., Schmeding, A.: Differentiable mappings on products with
  different degrees of differentiability in the two factors.
\newblock Expo. Math. \textbf{33}(2), 184--222 (2015).
\newblock \doi{10.1016/j.exmath.2014.07.002}

\bibitem{MR0177277}
Bastiani, A.: Applications diff\'erentiables et vari\'et\'es diff\'erentiables
  de dimension infinie.
\newblock J. Analyse Math. \textbf{13}, 1--114 (1964)

\bibitem{MR2540955}
Beattie, M.: A survey of {H}opf algebras of low dimension.
\newblock Acta Appl. Math. \textbf{108}(1), 19--31 (2009).
\newblock \doi{10.1007/s10440-008-9367-3}

\bibitem{MR2069671}
Bertram, W., Gl{\"o}ckner, H., Neeb, K.H.: Differential calculus over general
  base fields and rings.
\newblock Expo. Math. \textbf{22}(3), 213--282 (2004).
\newblock \doi{10.1016/S0723-0869(04)80006-9}

\bibitem{BDS16}
Bogfjellmo, G., Dahmen, R., Schmeding, A.: Character groups of {H}opf algebras
  as infinite-dimensional {L}ie groups.
\newblock Ann. Inst. Fourier (Grenoble) \textbf{66}(5), 2101--2155 (2016)

\bibitem{BDS17}
Bogfjellmo, G., Dahmen, R., Schmeding, A.: Overview of (pro-){L}ie group
  structures on {H}opf algebra character groups.
\newblock In: M.~Barbero, K.~Ebrahimi-Fard, D.M. de~Diego (eds.) {Discrete
  Mechanics, Geometric Integration and Lie-Butcher Series}, Springer
  Proceedings in Mathematics \& Statistics (2017)

\bibitem{BS16tame}
Bogfjellmo, G., Schmeding, A.: The tame {B}utcher group.
\newblock J. Lie Theory \textbf{26}, 1107--1144 (2016)

\bibitem{BS15}
Bogfjellmo, G., Schmeding, A.: The {L}ie group structure of the {B}utcher
  group.
\newblock Found. Comput. Math. \textbf{17}(1), 127--159 (2017).
\newblock \doi{10.1007/s10208-015-9285-5}

\bibitem{MR1728312}
Bourbaki, N.: Lie groups and {L}ie algebras. {C}hapters 1--3.
\newblock Elements of Mathematics (Berlin). Springer-Verlag, Berlin (1998).
\newblock Translated from the French, Reprint of the 1989 English translation

\bibitem{Brouder-04-BIT}
Brouder, C.: {Trees, renormalization and differential equations}.
\newblock BIT Num. Anal. \textbf{44}, 425--438 (2004)

\bibitem{1610.08468v1}
Bruned, Y., Hairer, M., Zambotti, L.: {Algebraic renormalisation of regularity
  structures} (2016).
\newblock \urlprefix\url{http://arxiv.org/abs/1610.08468v1}

\bibitem{Butcher72}
Butcher, J.C.: An algebraic theory of integration methods.
\newblock Math. Comp. \textbf{26}, 79--106 (1972)

\bibitem{MR2290769}
Cartier, P.: A primer of {H}opf algebras.
\newblock In: Frontiers in number theory, physics, and geometry. {II}, pp.
  537--615. Springer, Berlin (2007)

\bibitem{MR2371808}
Connes, A., Marcolli, M.: Noncommutative geometry, quantum fields and motives,
  \emph{American Mathematical Society Colloquium Publications}, vol.~55.
\newblock American Mathematical Society, Providence, RI; Hindustan Book Agency,
  New Delhi (2008)

\bibitem{MR0287271}
Floret, K.: Lokalkonvexe {S}equenzen mit kompakten {A}bbildungen.
\newblock J. Reine Angew. Math. \textbf{247}, 155--195 (1971)

\bibitem{MR1948922}
Gl{\"o}ckner, H.: Algebras whose groups of units are {L}ie groups.
\newblock Studia Math. \textbf{153}(2), 147--177 (2002).
\newblock \doi{10.4064/sm153-2-4}.
\newblock \urlprefix\url{http://dx.doi.org/10.4064/sm153-2-4}

\bibitem{MR1911979}
Gl\"ockner, H.: Infinite-dimensional {L}ie groups without completeness
  restrictions.
\newblock In: Geometry and analysis on finite- and infinite-dimensional {L}ie
  groups ({B}\polhk edlewo, 2000), \emph{Banach Center Publ.}, vol.~55, pp.
  43--59. Polish Acad. Sci. Inst. Math., Warsaw (2002).
\newblock \doi{10.4064/bc55-0-3}

\bibitem{MR2402519}
Gl\"ockner, H.: Instructive examples of smooth, complex differentiable and
  complex analytic mappings into locally convex spaces.
\newblock J. Math. Kyoto Univ. \textbf{47}(3), 631--642 (2007).
\newblock \doi{10.1215/kjm/1250281028}.
\newblock \urlprefix\url{http://dx.doi.org/10.1215/kjm/1250281028}

\bibitem{1601.02568v1}
Gl{\"o}ckner, H.: {Measurable regularity properties of infinite-dimensional
  {L}ie groups} (2015).
\newblock \urlprefix\url{http://arxiv.org/abs/1601.02568v1}

\bibitem{1208.0715v3}
Gl{\"o}ckner, H.: {Regularity properties of infinite-dimensional Lie groups,
  and semiregularity} (2015).
\newblock \urlprefix\url{http://arxiv.org/abs/1208.0715v3}

\bibitem{MR2997582}
Gl{\"o}ckner, H., Neeb, K.H.: When unit groups of continuous inverse algebras
  are regular {L}ie groups.
\newblock Studia Math. \textbf{211}(2), 95--109 (2012).
\newblock \doi{10.4064/sm211-2-1}.
\newblock \urlprefix\url{http://dx.doi.org/10.4064/sm211-2-1}

\bibitem{GNforth}
Gl{\"o}ckner, H., Neeb, K.H.: {Infinite-dimensional Lie Groups. General Theory
  and Main Examples} (2018).
\newblock Unpublished

\bibitem{MR1789831}
Gracia-Bond\'\i~a, J.M., V\'arilly, J.C., Figueroa, H.: Elements of
  noncommutative geometry.
\newblock Birkh\"auser Advanced Texts: Basler Lehrb\"ucher. [Birkh\"auser
  Advanced Texts: Basel Textbooks]. Birkh\"auser Boston, Inc., Boston, MA
  (2001).
\newblock \doi{10.1007/978-1-4612-0005-5}

\bibitem{HLW2006}
Hairer, E., Lubich, C., Wanner, G.: {Geometric Numerical Integration},
  \emph{{Springer Series in Computational Mathematics}}, vol.~31.
\newblock Springer Verlag ($^2$2006)

\bibitem{MR3274562}
Hairer, M.: A theory of regularity structures.
\newblock Invent. Math. \textbf{198}(2), 269--504 (2014).
\newblock \doi{10.1007/s00222-014-0505-4}.
\newblock \urlprefix\url{http://dx.doi.org/10.1007/s00222-014-0505-4}

\bibitem{MR0262773}
Hewitt, E., Ross, K.A.: Abstract harmonic analysis. {V}ol. {II}: {S}tructure
  and analysis for compact groups. {A}nalysis on locally compact {A}belian
  groups.
\newblock Die Grundlehren der mathematischen Wissenschaften, Band 152.
  Springer-Verlag, New York-Berlin (1970)

\bibitem{MR2337107}
Hofmann, K.H., Morris, S.A.: The {L}ie theory of connected pro-{L}ie groups,
  \emph{EMS Tracts in Mathematics}, vol.~2.
\newblock EMS, Z\"urich (2007).
\newblock \doi{10.4171/032}

\bibitem{MR3114697}
Hofmann, K.H., Morris, S.A.: The structure of compact groups, \emph{De Gruyter
  Studies in Mathematics}, vol.~25.
\newblock De Gruyter, Berlin (2013).
\newblock \doi{10.1515/9783110296792}.
\newblock A primer for the student---a handbook for the expert, Third edition,
  revised and augmented

\bibitem{HMax}
Hofmann, K.H., Morris, S.A.: Pro-{L}ie groups: A survey with open problems.
\newblock Axioms \textbf{4}, 294--312 (2015).
\newblock \doi{10.3390/axioms4030294}

\bibitem{MR632257}
Jarchow, H.: Locally convex spaces.
\newblock B. G. Teubner, Stuttgart (1981).
\newblock Mathematische Leitf\"aden. [Mathematical Textbooks]

\bibitem{keller1974}
Keller, H.: {Differential Calculus in Locally Convex Spaces}.
\newblock Lecture Notes in Mathematics 417. Springer Verlag, Berlin (1974)

\bibitem{MR3395225}
Kock, J.: Perturbative renormalisation for not-quite-connected bialgebras.
\newblock Lett. Math. Phys. \textbf{105}(10), 1413--1425 (2015).
\newblock \doi{10.1007/s11005-015-0785-7}

\bibitem{MR3526112}
K\"onig, W.: The parabolic {A}nderson model.
\newblock Pathways in Mathematics. Birkh\"auser/Springer, [Cham] (2016).
\newblock \doi{10.1007/978-3-319-33596-4}.
\newblock Random walk in random potential

\bibitem{KM97}
Kriegl, A., Michor, P.W.: {The convenient setting of global analysis},
  \emph{{Mathematical Surveys and Monographs}}, vol.~53.
\newblock AMS (1997)

\bibitem{MR1381692}
Majid, S.: Foundations of quantum group theory.
\newblock Cambridge University Press, Cambridge (1995).
\newblock \doi{10.1017/CBO9780511613104}

\bibitem{MR857807}
Mallios, A.: Topological algebras. {S}elected topics, \emph{North-Holland
  Mathematics Studies}, vol. 124.
\newblock North-Holland Publishing Co., Amsterdam (1986).
\newblock Notas de Matem\'atica [Mathematical Notes], 109

\bibitem{manchon}
Manchon, D.: Hopf algebras in renormalisation.
\newblock In: Handbook of algebra. {V}ol. 5, \emph{Handb. Algebr.}, vol.~5, pp.
  365--427. Elsevier/North-Holland, Amsterdam (2008).
\newblock \doi{10.1016/S1570-7954(07)05007-3}

\bibitem{MR3510021}
McLachlan, R.I., Modin, K., Munthe-Kaas, H., Verdier, O.: B-series methods are
  exactly the affine equivariant methods.
\newblock Numer. Math. \textbf{133}(3), 599--622 (2016).
\newblock \doi{10.1007/s00211-015-0753-2}.
\newblock \urlprefix\url{http://dx.doi.org/10.1007/s00211-015-0753-2}

\bibitem{MR2035107}
Michaelis, W.: Coassociative coalgebras.
\newblock In: Handbook of algebra, {V}ol. 3, pp. 587--788. North-Holland,
  Amsterdam (2003).
\newblock \doi{10.1016/S1570-7954(03)80072-4}.
\newblock \urlprefix\url{http://dx.doi.org/10.1016/S1570-7954(03)80072-4}

\bibitem{MR830252}
Milnor, J.: Remarks on infinite-dimensional {L}ie groups.
\newblock In: Relativity, groups and topology, {II} ({L}es {H}ouches, 1983),
  pp. 1007--1057. North-Holland, Amsterdam (1984)

\bibitem{MR0174052}
Milnor, J.W., Moore, J.C.: On the structure of {H}opf algebras.
\newblock Ann. of Math. (2) \textbf{81}, 211--264 (1965).
\newblock \doi{10.2307/1970615}.
\newblock \urlprefix\url{http://dx.doi.org/10.2307/1970615}

\bibitem{MR3485151}
Murua, A., Sanz-Serna, J.M.: Computing normal forms and formal invariants of
  dynamical systems by means of word series.
\newblock Nonlinear Anal. \textbf{138}, 326--345 (2016).
\newblock \doi{10.1016/j.na.2015.10.013}.
\newblock \urlprefix\url{http://dx.doi.org/10.1016/j.na.2015.10.013}

\bibitem{MR2261066}
Neeb, K.H.: Towards a {L}ie theory of locally convex groups.
\newblock Jpn. J. Math. \textbf{1}(2), 291--468 (2006).
\newblock \doi{10.1007/s11537-006-0606-y}

\bibitem{MR0342978}
Schaefer, H.H.: Topological vector spaces.
\newblock Springer-Verlag, New York-Berlin (1971).
\newblock Third printing corrected, Graduate Texts in Mathematics, Vol. 3

\bibitem{MR0448350}
Swan, R.G.: Topological examples of projective modules.
\newblock Trans. Amer. Math. Soc. \textbf{230}, 201--234 (1977).
\newblock \doi{10.2307/1997717}.
\newblock \urlprefix\url{http://dx.doi.org/10.2307/1997717}

\bibitem{Swe69}
Sweedler, M.E.: Hopf algebras.
\newblock Mathematics Lecture Note Series. W. A. Benjamin, Inc., New York
  (1969)

\end{thebibliography}
\end{document}